\documentclass[11pt]{amsart}
\usepackage{geometry}                
\usepackage{graphicx,caption,subcaption}
\usepackage{amssymb}
\usepackage{epstopdf}
\usepackage{hyperref}
\usepackage{color}
\newcommand{\Z}{\mathbb{Z}}
\newcommand{\R}{\mathbb{R}}
\newcommand{\N}{\mathbb{N}}
\newcommand{\C}{\mathbb{C}}

\numberwithin{equation}{section}
\numberwithin{figure}{section}

\theoremstyle{plain} 
\newtheorem{theorem}{Theorem}[section]

\newtheorem{proposition}[theorem]{Proposition}

\theoremstyle{definition} 
\newtheorem{definition}[theorem]{Definition}

\title{A note on balance equations for doubly periodic minimal surfaces}

\author
{Peter Connor}
\address{Peter Connor\\Department of Mathematical Sciences\\Indiana University South Bend\\South
Bend\\IN 46634\\USA}


\date{\today}

\begin{document}

\begin{abstract}
Most known examples of doubly periodic minimal surfaces in $\R^3$ with parallel ends limit as a foliation of $\R^3$ by horizontal noded planes, with the location of the nodes satisfying a set of balance equations.  Conversely, for each set of points providing a balanced configuration, there is a corresponding three-parameter family of doubly periodic minimal surfaces.    In this note we derive a differential equation that is equivalent to the balance equations for doubly periodic minimal surfaces.  This allows for the generation of many more solutions to the balance equations, enabling the construction of increasingly complicated surfaces.
\end{abstract}

\maketitle

\section{Introduction}
Many doubly periodic minimal surfaces in $\R^3$ with parallel ends limit as a foliation of parallel planes connected by tiny catenoid necks that shrink to nodes at the limit.  This was the case with the first examples of genus one constructed by Karcher \cite{ka4} and Meeks and Rosenberg \cite{mr3}, and of genus two constructed by Wei \cite{wei2}.  It was also the case with the surfaces constructed in \cite{cw1}, in which the author and Weber proved that for any genus $g\geq 1$ and any even number $N\geq2$ there are three-parameter families of embedded doubly periodic minimal surfaces of genus $g$ and $2N$ parallel ends.  Each family of surfaces is constructed in a neighborhood of a noded limit.  Given a set of points in the complex plane that satisfy a set of balance equations, theorem 2.1 in \cite{cw1} provides a three-parameter family of surfaces that geometrically look like parallel planes connected by periodically placed catenoid necks, with the location of the necks given by the solutions to the balance equations.  See figure \ref{figure:(1,n)}.

Solving the balance equations proved difficult due to the fact that there are many equivalent solutions by permuting the locations of the nodes at a given level.  Employing techniques used by Traizet in \cite{tr4,tr8} to find balance configurations for minimal surfaces with finite total curvature, the balance equations can be combined into a differential equation that mitigates this difficulty.  This note demonstrates how to do so with the doubly periodic balance equations.

In section 2, we discuss forces, balance equations, and the known balanced configurations for doubly periodic minimal surfaces.  In section 3, we prove that the balance equations are equivalent to a second order differential equation.  In section 4, we examine configurations of type $(2,n)$.  In section 5, we examine configurations of type $(3,4)$, which is the smallest configuration with no non-trivial symmetries.

\begin{figure}[h]
	\centerline{ 
		\includegraphics[height=2in]{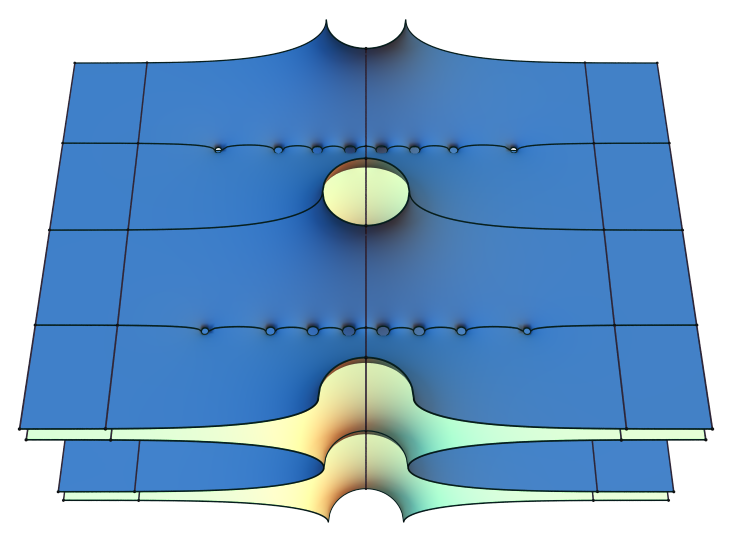}
		\includegraphics[height=2in]{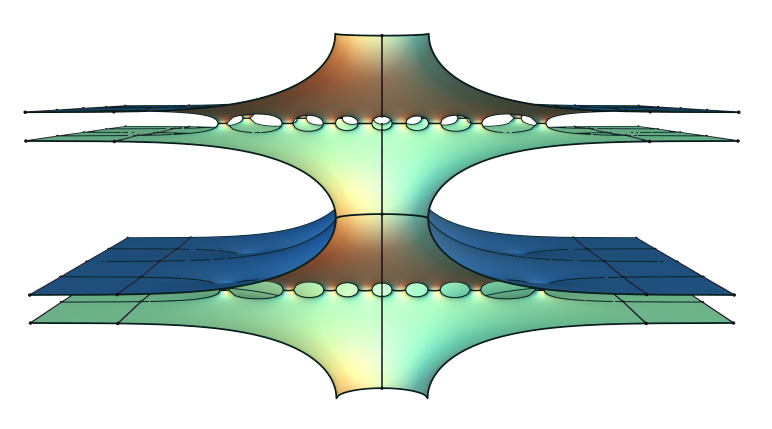}
			}
	\caption{Two views of a genus 8 doubly periodic minimal surface}
	\label{figure:(1,n)}
\end{figure}

\section{Forces and Balance Equations}
A doubly periodic minimal surface $M$ in $\R^3$ is invariant under two linearly independent translations given by a two dimensional lattice $\Lambda$.  There is a corresponding minimal surface $\tilde{M}$ in the quotient space $\R^3/\Lambda$, from which one can recover $M$.  Assume that the generators of $\Lambda$ are the vector $(0,2\pi,0)$ and a non-horizontal vector and that the ends of the surface have vertical limiting normal.  Then, each level of the quotient surface has domain $\C/(2\pi i\Z)$.  For convenience of calculations, this is identified with $\C^*=\C-\{0\}$ via the exponential map.

Consider $N$ copies of $\mathbb{C}^*$, labeled $\mathbb{C}_k^*$ for $k=1,\ldots,N$, which correspond to the different levels of the surface.  The ends of the surface are placed at $0_k=0$ and $\infty_k=\infty$ in $\overline{\C}_k$.  On each $\mathbb{C}_k^*$, place $n_k$ points $p_{k,1},\ldots,p_{k,n_k}$.  Extend this  definition of $p_{k,i}$ for any integer $k$ by making it  periodic in the sense that $p_{k+N,i}=p_{k,i}$ for $k=1,\ldots,N$ and $i=1,\ldots,n_k$, with $n_{k+N}=n_k$.  Each point $\tilde{p}_{k,i}=\log p_{k,i}$ corresponds to the location of a catenoid shaped neck between the $k-1$ an $k$ levels of the surface.  

Given a family of doubly periodic minimal surfaces that limits as a foliation of noded planes, the location of the nodes must satisfy a balancing condition given in terms of the following force equations.

\begin{definition}
The {\it force} exerted on $p_{k,i}$ by the other points in $\{p_{k,i}\}$ is defined by
\[
F_{k,i}:=\sum_{j \neq i}\frac{p_{k,i}+p_{k,j}}{n_k^2(p_{k,i}-p_{k,j})}-\sum_{j=1}^{n_{k+1}}\frac{p_{k,i}+p_{k+1,j}}{2 n_kn_{k+1}\left(p_{k,i}-p_{k+1,j}\right)}-\sum_{j=1}^{n_{k-1}}\frac{p_{k,i}+p_{k-1,j}}{2 n_kn_{k-1}\left(p_{k,i}-p_{k-1,j}\right)}.
\]
The equations $F_{k,i}=0$ are referred to as {\it balance equations}.
\end{definition}

\begin{definition}
The configuration $\{p_{k,i}\}$ is called a \textit{balanced configuration} if $F_{k,i}=0$ for $k=1,\ldots,N$ and $i=1,\ldots,n_k$.  It is a balanced configuration of type $(n_1,n_2,\ldots,n_N)$.
\end{definition}

\begin{definition}
A configuration $\{p_{k,i}\}$ is said to be {\it non-degenerate} if the Jacobian matrix $\partial F_{k,i}/\partial p_{j,h}$ has complex rank $m-1$, where $\displaystyle m=\sum_{k=1}^N n_k$.
\end{definition}

The Jacobian matrix can't have full rank $2m$ because $$\sum_{k=1}^N\sum_{i=1}^{n_k}F_{k,i}=0.$$  This holds whether or not the configuration $\{p_{k,i}\}$ is balanced. 

Theorem 2.1 from \cite{cw1} states that, given a non-degenerate balanced configuration $\{p_{k,i}\}$, there exists a three-parameter family of embedded doubly periodic minimal surfaces that limit as a foliation of $\R^3$ by horizontal noded planes.  Each quotient surface has genus 
\[
g=1+\sum_{k=1}^N(n_k-1)
\]
and $2N$ ends asymptotic to flat cylinders, two at each of the $N$ levels.  There are $n_k$ catenoid necks joining the $k-1$ and $k$ levels, with the horizontal position of the necks given by the terms $\tilde{p}_{k,i}=\log p_{k,i}$, $i=1,2,\ldots, n_k$.

When the surfaces are viewed in $\R^3$, there are infinitely many levels, with the height of level $N+k$ equal to the sum of the heights of level $N$ and level $k$.  Also, there are infinitely many periodically placed necks between successive levels, with the horizontal locations $\tilde{p}_{k,i}$ of the necks periodic with respect to the translation vector $(0,2\pi,0)$. 

Theorem 2.1 was proven by constructing the Weierstrass representation for the desired surfaces in a neighborhood of a noded limit and solving the period problem on the noded limit.  Part of solving the period problem is having a balanced configuration.  The configuration being non-degenerate allows the use of the implicit function theorem to solve the period problem in an open neighborhood of the noded limit.

In \cite{cw1}, non-degenerate balanced configurations were shown to exist when $N=2$, $n_1=1$, and $n_2=n$ for any $n\in\mathbb{N}$.  On each quotient surface, this data corresponds to two levels, each with two Scherk ends.  Between the levels, there are catenoid necks.  From level one to level two, there are $n$ necks.  From level two to level three (level one in the quotient), there is one neck.  We refer to these as $(1,n)$ configurations, designating two levels with $1$ and $n$ necks between successive levels.  The surface in figure \ref{figure:(1,n)} corresponds to a $(1,8)$ balanced configuration.  

For each $n\in\N$ there is only one $(1,n)$ balanced configuration.  The location of the nodes are $p_{1,1}$ and $p_{2,k}$, $k=1,2,\ldots,n$, with $p_{1,1}=1$ and the $p_{2,k}$ corresponding to roots of the polynomial
\[
p_n(z)=\sum_{k=0}^n{n \choose k}^2z^k.
\]

It was also proven that sequences of this type of configuration can be concatenated to produce a new non-degenerate balanced configuration.  If there exist non-degenerate balanced configurations of type $(1,n_j)$  for $j=1,2,\ldots,m$ then they can be combined to create a non-degenerate balanced configuration of type $(1,n_1,1,n_2,\ldots,1,n_m)$, with corresponding embedded, doubly periodic minimal surface with $2m$ levels and the number of necks between successive levels alternating between $1$ and the integers $n_j$.  

Two $(2,3)$ balanced configurations were discovered, which led to the question of whether there are always balanced configurations of the form $(m,n)$ with $1\leq m\leq n$.  Numerical evidence indicates that the number of balanced configurations of a fixed type $(m,n)$ increases as $m$ increases.  The locations of the necks of the surface in figure \ref{figure:(2,13)} are given by one of the seven balanced configurations of type $(2,13)$.

\begin{figure}[h]
	\centerline{ 
		\includegraphics[height=2in]{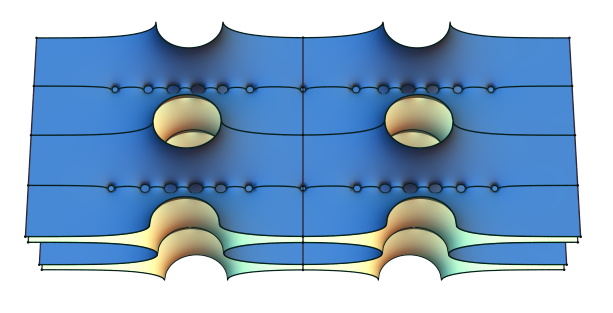}
			}
	\caption{Surface corresponding to a $(2,13)$ balanced configuration}
	\label{figure:(2,13)}
\end{figure}

\section{An alternative to the balance equations}
The balance equations corresponding to more complicated configurations such as those of type $(n_1,n_2)$ with $1<n_1<n_2$ are very difficult to solve algebraically.  In \cite{tr4,tr8}, Traizet combined a set of balance equations for minimal surfaces in $\R^3$ with finite total curvature into one differential equation.  One solution of the differential equation corresponds to many equivalent balanced configurations by permutation of the nodes at each level, and so it is much easier to find balanced configurations by solving the corresponding differential equation.  We use Traizet's method to find a differential equation corresponding to the balance equations for doubly periodic minimal surfaces.   

\begin{theorem}
Let $N$ be an even positive integer, $n_1,n_2,\ldots, n_N\in\N$, and suppose $\{p_{k,i}\}$ is a configuration such that the $p_{k,i}$ are distinct.  Let   
\[
P_k(z)=\prod_{i=1}^{n_k}(z-p_{k,i}),\,P(z)=\prod_{k=1}^N P_k(z)
\]
and
\[
Q(z)=\sum_{k=1}^N\left(\frac{zP_k''(z)P(z)}{n_k^2P_k(z)}-\frac{zP_k'(z)P_{k+1}'(z)P(z)}{n_kn_{k+1}P_k(z)P_{k+1}(z)}+\frac{P_k'(z)P(z)}{n_k^2P_k(z)}\right).
\]
Then the configuration $\{p_{k,i}\}$ is balanced if and only if $Q(z)\equiv 0$.
\end{theorem}

\begin{proof}
An equivalent expression for the force $F_{k,i}$ is given by
\[
\frac{2p_{k,i}}{n_k^2}\sum_{j \neq i}\frac{1}{p_{k,i}-p_{k,j}}-\frac{p_{k,i}}{n_kn_{k+1}}\sum_{j=1}^{n_{k+1}}\frac{1}{p_{k,i}-p_{k+1,j}}-\frac{p_{k,i}}{n_{k-1}n_k}\sum_{j=1}^{n_{k-1}}\frac{1}{p_{k,i}-p_{k-1,j}}+\frac{1}{n_k^2}.
\]
Since the $p_{k,i}$ are distinct for each $k$,
\[
\frac{P_k''(p_{k,i})}{P_k'(p_{k,i})}=\sum_{j\neq i}\frac{2}{p_{k,i}-p_{k,j}},
\]

\[
\frac{P_{k+1}'(p_{k,i})}{P_{k+1}(p_{k,i})}=\sum_{j=1}^{n_{k+1}}\frac{1}{p_{k,i}-p_{k+1,j}},
\]

\[
\frac{P_{k-1}'(p_{k,i})}{P_{k-1}(p_{k,i})}=\sum_{j=1}^{n_{k-1}}\frac{1}{p_{k,i}-p_{k-1,j}},
\]
and the force equations can be rewritten in terms of the polynomials $P_k$:
\[
F_{k,i}=\frac{p_{k,i}P_k''(p_{k,i})}{n_k^2P_k'(p_{k,i})}-\frac{p_{k,i}P_{k+1}'(p_{k,i})}{n_kn_{k+1}P_{k+1}(p_{k,i})}-\frac{p_{k,i}P_{k-1}'(p_{k,i})}{n_kn_{k-1}P_{k-1}(p_{k,i})}+\frac{1}{n_k^2}
\]
Substituting $z$ for $p_{k,i}$ and multiplying by 
\[
\frac{P_k'(z)P(z)}{P_k(z)}
\]
we get the polynomial
\[
Q_k(z)=\frac{zP_k''(z)P(z)}{n_k^2P_k(z)}-\frac{zP_k'(z)P_{k+1}'(z)P(z)}{n_kn_{k+1}P_k(z)P_{k+1}(z)}-\frac{zP_{k-1}'(z)P_k'(z)P(z)}{n_kn_{k-1}P_{k-1}(z)P_k(z)}+\frac{P_k'(z)P(z)}{n_k^2P_k(z)},
\]
and for each $i=1,2,\ldots, n_k$, $F_{k,i}=0$ if and only if $Q_k(p_{k,i})=0$.  

Then,
\[
Q(z)=\sum_{k=1}^N\left(\frac{zP_k''(z)P(z)}{n_k^2P_k(z)}-\frac{zP_k'(z)P_{k+1}'(z)P(z)}{n_kn_{k+1}P_k(z)P_{k+1}(z)}+\frac{P_k'(z)P(z)}{n_k^2P_k(z)}\right)
\]
is a polynomial with degree less than $m=\sum n_k$, and $Q(p_{k,i})=Q_k(p_{k,i})$ for $i=1,2,\ldots, n_k$ and $k=1,2,\ldots,N$.

If $Q(z)\equiv 0$ then $Q_k(p_{k,i})=0$ and $F_{k,i}=0$ for $i=1,2,\ldots,n_k$ and $k=1,2,\ldots,N$, and so the configuration $\{p_{k,i}\}$ is balanced.  If the configuration $\{p_{k,i}\}$ is balanced then $Q(p_{k,i})=Q_k(p_{k,i})=F_{k,i}=0$.  Thus, $Q(z)$ has degree less than $m$ and at least $m$ distinct roots, and so $Q(z)\equiv 0$.
\end{proof}

Note that if we re-express
\[
P_k(x)=\sum_{i=0}^{n_k}a_{k,i}z^i
\]
then the $Q(z)\equiv 0$ is a system of at most $m-1$ equations with $m$ variables $a_{k,i}$.  

\subsection{Configurations of type $(n_1,n_2)$}
If $N$ =2 then, after multiplying by $n_1^2n_2^2$, $Q(z)$ is given by
\[
n_2^2zP_1''(z)P_2(z)+n_1^2zP_2''(z)P_1(z)-2n_1n_2zP_1'(z)P_2'(z)+n_2^2P_1'(z)P_2(z)+n_1^2P_2'(z)P_1(z).
\]

With some extra assumptions, the non-degeneracy of configurations of type $(n_1,n_2)$ is guaranteed.
\begin{proposition}
If $p_{k,i}\in\R$ with $p_{1,i}>0$ for $i=1,2,\ldots,n_1$ and $p_{2,i}<0$ for $i=1,2,\ldots,n_2$ then the configuration $\{p_{k,i}\}$ is non-degenerate.
\label{prop:nondegen}
\end{proposition}
\begin{proof}
If $p_{k,i}\in\R$ with $p_{1,i}>0$ for $i=1,2,\ldots,n_1$ and $p_{2,i}<0$ for $i=1,2,\ldots,n_2$ then the Jacobian matrix $\partial F_{k,i}/\partial p_{j,h}$ is a $(n_1+n_2)\times(n_1+n_2)$ matrix, and it is easy to see that the submatrix obtained by removing the last row and column is strongly diagonally dominant.  Hence, the Jacobian matrix has rank $n_1+n_2-1$, and the configuration $p_{k,i}$ is non-degenerate. 
\end{proof}
Otherwise, the non-degeneracy of a given balanced configuration can be checked on a case by case basis.
\section{Configurations of type $(2,n)$}
Consider the case when $N=2$, $n_1=2$, and $n_2=n\geq 2$.  After rescaling and translating, we can assume that $p_{1,2}=1/p_{1,1}$.  Then
\[
P_1(z)=(z-p_{1,1})(z-p_{1,2})=z^2-\alpha z+1,
\] 
\[
P_2(z)=\prod_{i=1}^n(z-p_{2,i})=\sum_{i=0}^n a_iz^i,
\]
and
\[
Q(z)=4(z^3-\alpha z^2+z)P_2''(z)+4\left((1-2n)z^2+(\alpha n-\alpha)z+1\right)P_2'(z)+n^2(4z-\alpha)P_2(z).
\]

Finding balanced  $(2,n)$ configurations corresponds to finding a $\alpha\in\R$ and polynomial $P_2(z)$ such that $Q(z)\equiv 0$ and the roots of $P_1(z)P_2(z)$ are distinct.   

In this case, $Q(z)$ is a polynomial of degree at most $n+1$.  If we re-express 
\[
Q(z)=\sum_{i=0}^{n+1}b_i z^i
\]
then
\[
b_k=4a_{k-1}\left(k-n-1\right)^2-\alpha a_k(2k-n)^2+4a_{k+1}(k+1)^2
\]
with $a_k=0$ for $k>n$.

We want $Q(z)\equiv 0$, which is the same as $b_k=0$ for $0\leq k\leq n+1$, and 

\[
\begin{split}
b_{n+1}&=0\Leftrightarrow4a_n(n+1-n-1)^2=0\\
\end{split}
\]
and
\[
b_k=0\Leftrightarrow a_{k-1}=\frac{\alpha a_k(2k-n)^2-4a_{k+1}(k+1)^2}{4(k-n-1)^2}
\]
for $k=1,2,\ldots,n$.
Starting with $a_n=1$ and $\displaystyle a_{n-1}=\frac{\alpha n^2}{4}$, we can recursively define $a_{n-k}$ for $k=1,2,\ldots,n$.  Each $a_k$ is a polynomial with respect to $\alpha$ of degree at most $n-k$, call them $a_k=A_k(\alpha)$.

Thus, $\{p_{k,i}\}$ provides a balanced $(2,n)$ configuration if 
\[
P_1P_2(z)=\left(z^2-\alpha z+1\right)\left(\sum_{i=0}^nA_k(\alpha)z^i\right)
\]
has distinct roots.  This can be checked on a case by case basis.  Numerical evidence suggests that for each $n\in\N$ there is one balanced configuration of type $(2,2n+1)$ with $p_{1,1},p_{1,2}>0$ and $p_{2,i}<0$ for $i=1,2,\ldots,n$.  By proposition \ref{prop:nondegen}, this configuration is non-degenerate.  The non-degeneracy of other examples can be checked on a case by case basis.

\subsection{(2,4) Balanced Configurations}
If $n=4$ then $Q(z)\equiv 0$ when
\[
\begin{split}
b_4&=0\Leftrightarrow a_3=4\alpha\\
b_3&=0\Leftrightarrow a_2=\alpha^2-4\\
b_2&=0\Leftrightarrow a_1=-4\alpha\\
b_1&=0\Leftrightarrow a_0=-\frac{1}{2}(\alpha^2-2)\\
b_0&=0\Leftrightarrow \frac{1}{8}\alpha(\alpha^2-4)=0,
\end{split}
\]
and $\frac{1}{8}\alpha(\alpha^2-4)=0$ has roots $0$ and $\pm 2$.  However, $\pm 2$ don't work because then $P_1(z)$ has repeated root $z=1$ or $z=-1$.  If $\alpha=0$ then
\[
P_1(z)=z^2+1,\,P_2(z)=z^4-4z^2+1.
\] 
Then $P_1(z)$ has roots $\pm i$ and $P_2(z)$ has roots $\pm\sqrt{2-\sqrt{3}},\pm\sqrt{2+\sqrt{3}}$.  Hence, the nodes are located at 
\[
\pm\frac{\pi}{2}i, \frac{1}{2}\log(2+\sqrt{3}), \frac{1}{2}\log(2-\sqrt{3}), \frac{1}{2}\log(2+\sqrt{3})+\pi i,\frac{1}{2}\log(2-\sqrt{3})+\pi i.
\]  

However, the $(1,2)$ configuration has $p_{1,1}=1$, $p_{2,1}=-2+\sqrt{3}$, and $p_{2,2}=-2-\sqrt{3}$, with the location of the nodes 
\[
0,\log(2+\sqrt{3})+\pi i,\log(2-\sqrt{3})+\pi i.
\]
Thus, if we rescale the $(2,4)$ configuration by $2$ and translate by $\pi i$, we get the $(1,2)$ configuration.  See figure \ref{figure:(2,4)config}.

\begin{figure}[h]
    \centering
    \begin{subfigure}[b]{0.48\textwidth}
       \centering
        \includegraphics[width=.8\textwidth]{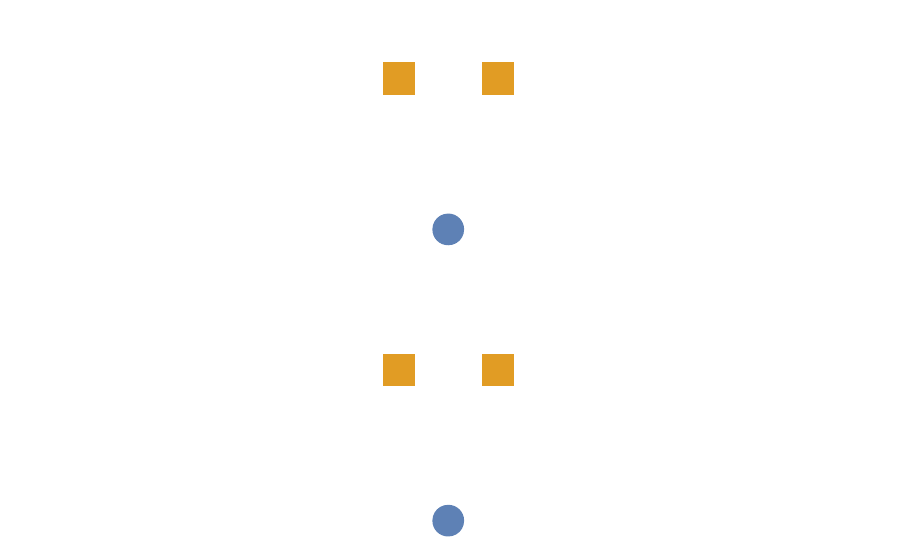}
        \caption{$(2,4)$ balanced configuration}
        \label{figure:(2,4)}
    \end{subfigure}
    ~ 
    \begin{subfigure}[b]{0.48\textwidth}
    	    \centering
        \includegraphics[width=.8\textwidth]{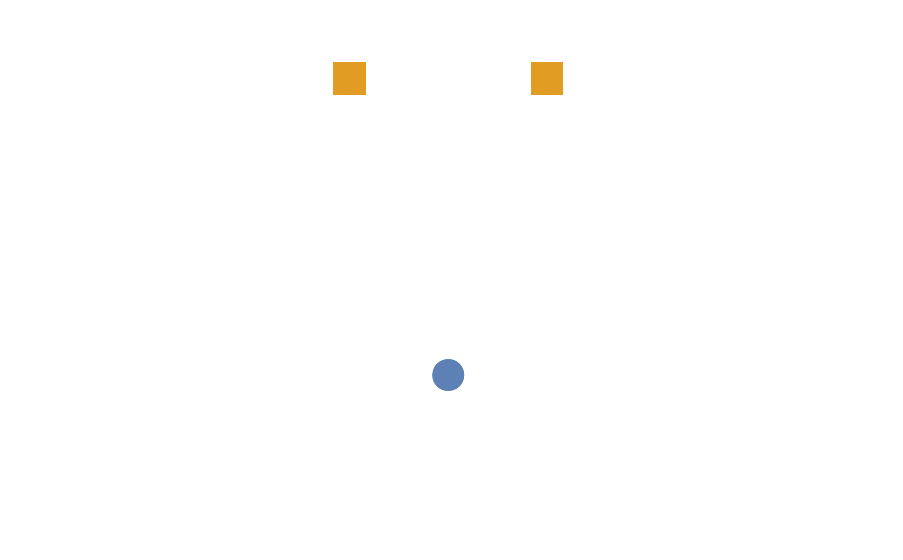}
        \caption{$(1,2)$ balanced configuration}
        \label{figure:(1,2)}
    \end{subfigure}
       \caption{The circles and squares represent the nodes at levels one and two, respectively.}\label{figure:(2,4)config}
\end{figure}

\subsection{(2,5) Balanced Configurations}
If $n=5$ then $Q(z)\equiv 0$ when
\[
\begin{split}
b_5&=0\Leftrightarrow a_4=\frac{25\alpha}{4}\\
b_4&=0\Leftrightarrow a_3=\frac{25(9\alpha^2-16)}{64}\\
b_3&=0\Leftrightarrow a_2=\frac{25\alpha(9\alpha^2-1040)}{2304}\\
b_2&=0\Leftrightarrow a_1=\frac{25(9\alpha^4-12704\alpha^2+20736)}{147456}\\
b_1&=0\Leftrightarrow a_0=\frac{\alpha(81\alpha^4-123552\alpha^2+1251584)}{589824}\\
b_0&=0\Leftrightarrow -\frac{(9\alpha^3-324\alpha^2-1040\alpha+576)(9\alpha^3+324\alpha^2-1040\alpha-576)}{2359296}=0.
\end{split}
\]
Because of the symmetries of the solutions, there are three balanced configurations corresponding to the positive solutions to the $b_0=0$ equation: $\alpha\approx0.48233788$, $\alpha\approx3.40867116$, or $\alpha\approx38.92633327$.  See figures \ref{figure:(2,5)config} and \ref{figure:(2,5)}.

\begin{figure}[h]
    \centering
    \begin{subfigure}[b]{0.3\textwidth}
       \centering
        \includegraphics[width=\textwidth]{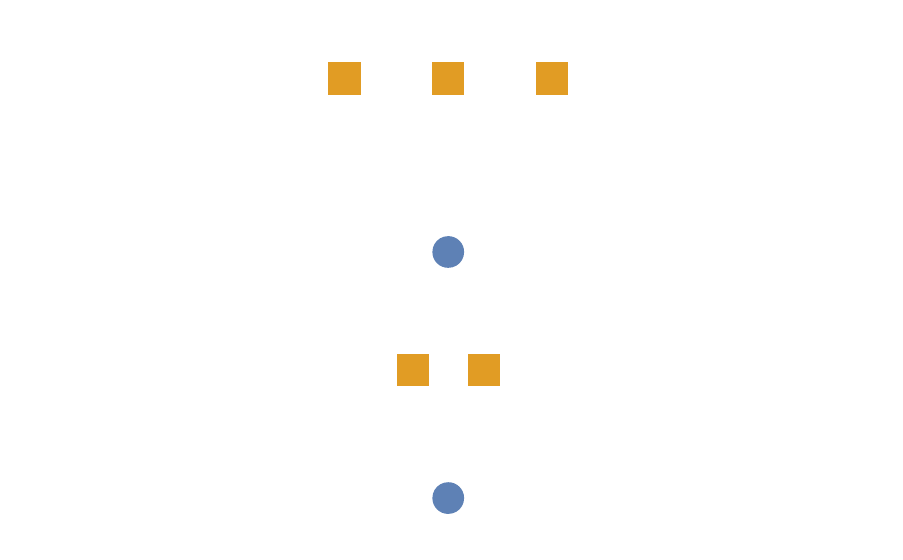}
        \caption{$\alpha\approx 0.48233788$}
        \label{figure:(2,5)a}
    \end{subfigure}
    ~ 
    \begin{subfigure}[b]{0.3\textwidth}
    	    \centering
        \includegraphics[width=\textwidth]{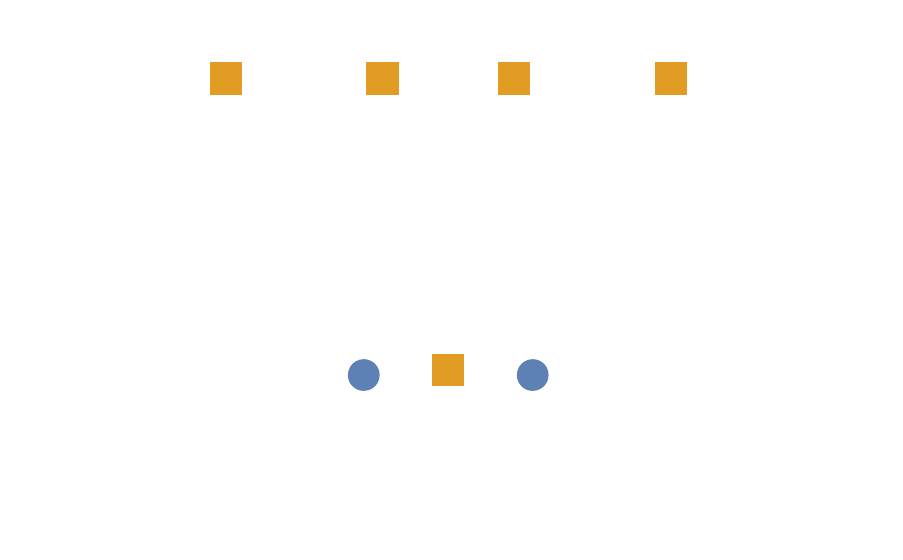}
        \caption{$\alpha\approx 3.40867116$}
        \label{figure:(2,5)b}
    \end{subfigure}
    ~ 
    \begin{subfigure}[b]{0.3\textwidth}
    	    \centering
        \includegraphics[width=\textwidth]{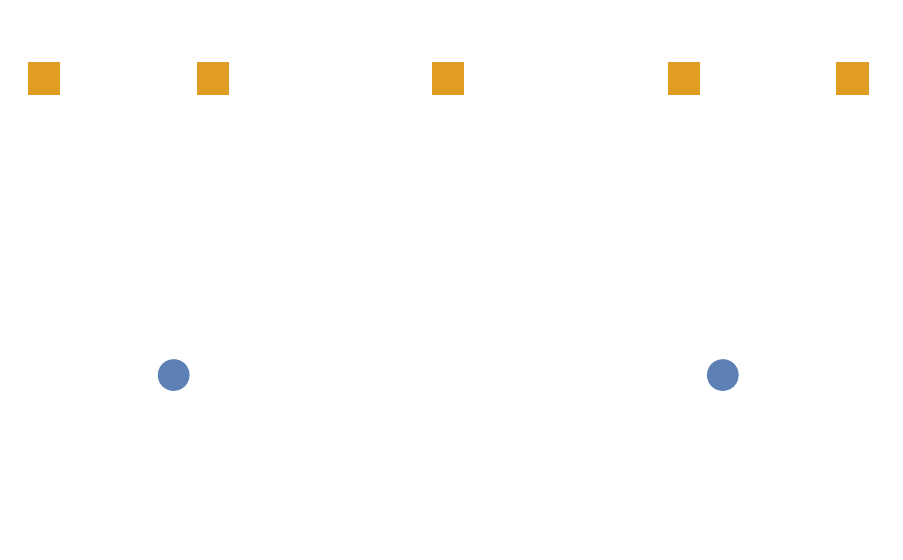}
        \caption{$\alpha\approx 38.92633327$}
        \label{figure:(2,5)c}
    \end{subfigure}
    \caption{$(2,5)$ balanced configurations}\label{figure:(2,5)config}
\end{figure}

\begin{figure}[h]
	\centerline{ 
		\includegraphics[height=2in]{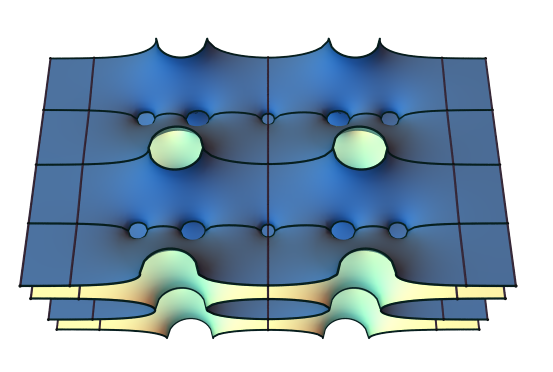}
			}
	\caption{Surface corresponding to the $(2,5)$ balanced configuration in figure \ref{figure:(2,5)c}}
	\label{figure:(2,5)}
\end{figure}

\subsection{(2,6) Balanced Configurations}
If $n=6$ then $Q(z)\equiv 0$ when
\[
\begin{split}
b_6&=0\Leftrightarrow a_5=9\alpha\\
b_5&=0\Leftrightarrow a_4=9(\alpha^2-1)\\
b_4&=0\Leftrightarrow a_3=\alpha(\alpha^2-26)\\
b_3&=0\Leftrightarrow a_2=-9(\alpha^2-1)\\
b_2&=0\Leftrightarrow a_1=-\frac{9}{25}(2\alpha^3-27\alpha)\\
b_1&=0\Leftrightarrow a_0=\frac{1}{25}(-2\alpha^4+52\alpha^2-25)\\
b_0&=0\Leftrightarrow \frac{1}{50}\alpha(\alpha^2-26)(\alpha^2-1)=0.
\end{split}
\]

So, $\alpha$ can be $0$, $\pm 1$, or $\pm \sqrt{26}$.  There are only two new configurations.  The $\alpha=0$ configuration is equivalent to the $(1,3)$ configuration, by a factor of $2$, the $\alpha=-\sqrt{26}$ configuration is equivalent to the $\alpha=\sqrt{26}$ configuration, by a translation of $\pi i$, and the $\alpha=-1$ configuration is equivalent to the $\alpha=1$ configuration, by a translation of $\pi i$.

\begin{figure}[h]
    \centering
     \begin{subfigure}[b]{0.3\textwidth}
    	    \centering
        \includegraphics[width=\textwidth]{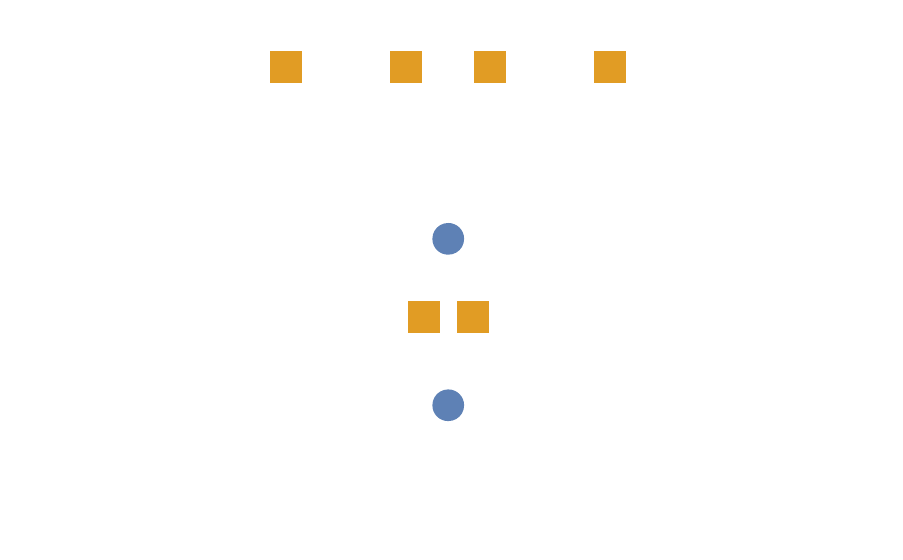}
        \caption{$\alpha=1$}
        \label{figure:(2,6)b}
    \end{subfigure}
    ~ 
    \begin{subfigure}[b]{0.3\textwidth}
    	    \centering
        \includegraphics[width=\textwidth]{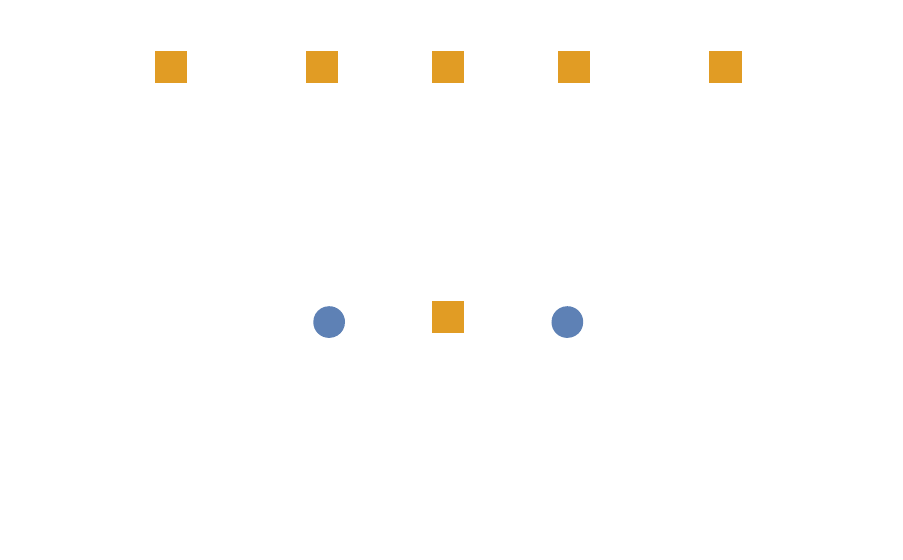}
        \caption{$\alpha=\sqrt{26}$}
        \label{figure:(2,6)c}
    \end{subfigure}
    \caption{$(2,6)$ balanced configurations}\label{figure:(2,6)config}
\end{figure}

\subsection{(2,7) Balanced Configurations}
If $n=7$ then $Q(z)\equiv 0$ when
\[
\begin{split}
b_7&=0\Leftrightarrow a_6=\frac{49\alpha}{4}\\
b_6&=0\Leftrightarrow a_5=\frac{49\alpha(25\alpha^2-16)}{64}\\
b_5&=0\Leftrightarrow a_4=\frac{49\alpha(25\alpha^2-272)}{256}\\
b_4&=0\Leftrightarrow a_3=\frac{49(25\alpha^4-10272\alpha^2+6400)}{16384}\\
b_3&=0\Leftrightarrow a_2=\frac{49\alpha(25\alpha^4-112672\alpha^2+1120512)}{1638400}\\
b_2&=0\Leftrightarrow a_1=\frac{49(25\alpha^6-122672\alpha^4+5229312\alpha^2-2560000)}{26214400}\\
b_1&=0\Leftrightarrow a_0=\frac{\alpha(625\alpha^6-3073200\alpha^4+159576832\alpha^2-350851072)}{104857600}\\
b_0&=0\Leftrightarrow 25\alpha^4-1600\alpha^3-10272\alpha^2+17408\alpha+6400=0\,\,\,\text{or}\\
&\hspace{.56in} 25\alpha^4+1600\alpha^3-10272\alpha^2-17408\alpha+6400=0.
\end{split}
\]
%
Because of the symmetries of the solutions, there are four balanced configurations corresponding to the positive solutions to the $b_0=0$ equation: $\alpha\approx0.312754$, $\alpha\approx1.65533$, $\alpha\approx7.08968$, or $\alpha\approx69.7471$.  See figures \ref{figure:(2,7)} and  \ref{figure:(2,7)config}.

\begin{figure}[h]
	\centerline{ 
		\includegraphics[height=2in]{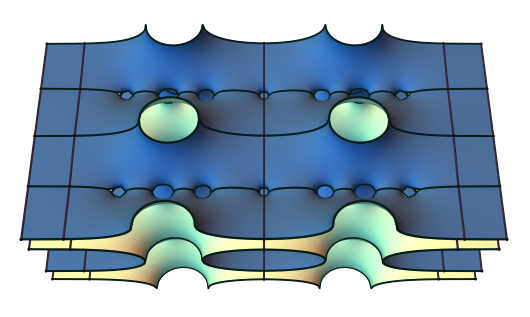}
			}
	\caption{Surface corresponding to the $(2,7)$ balanced configuration in figure \ref{figure:(2,7)d}}
	\label{figure:(2,7)}
\end{figure}

\begin{figure}[h]
    \centering
    \begin{subfigure}[b]{0.48\textwidth}
       \centering
        \includegraphics[width=\textwidth]{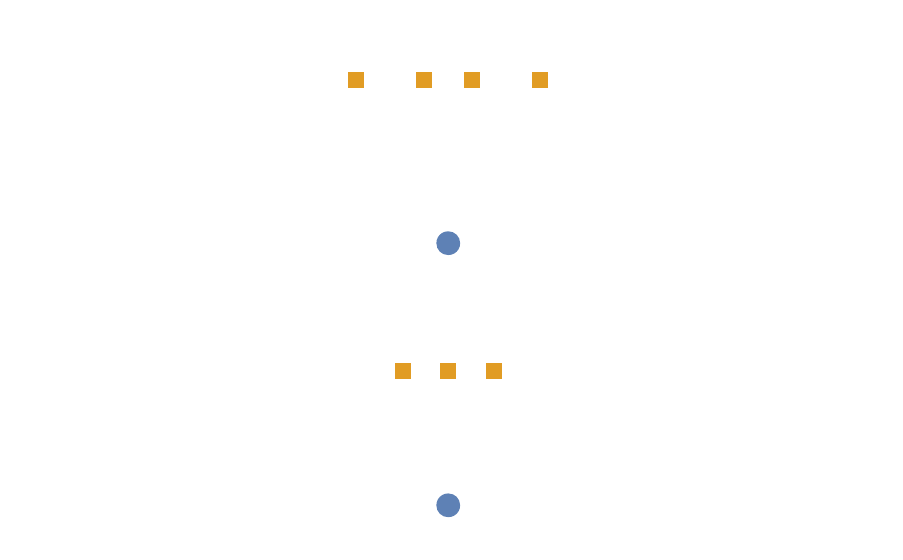}
        \caption{$\alpha\approx 0.312754$}
        \label{figure:(2,7)a}
    \end{subfigure}
    ~ 
    \begin{subfigure}[b]{0.48\textwidth}
    	    \centering
        \includegraphics[width=\textwidth]{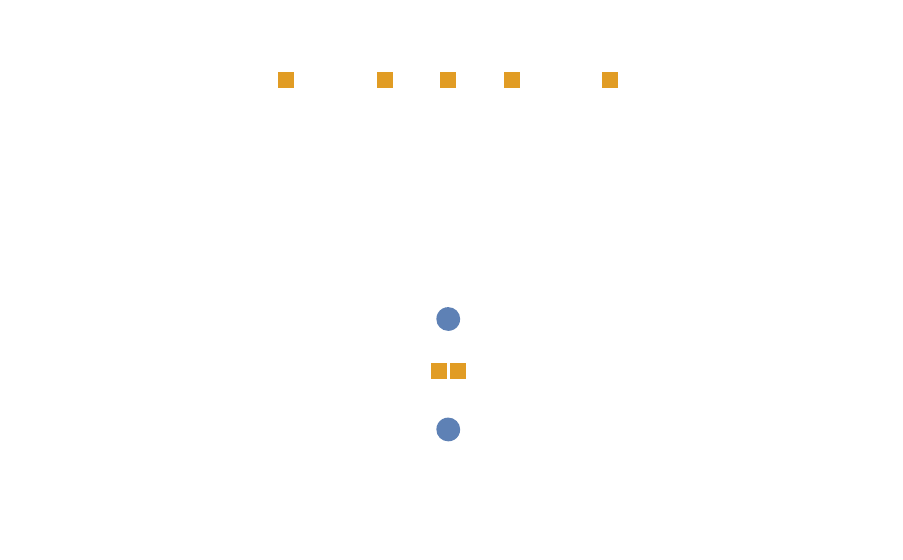}
        \caption{$\alpha\approx 1.65533$}
        \label{figure:(2,7)b}
    \end{subfigure}

    ~ 
    \begin{subfigure}[b]{0.48\textwidth}
    	    \centering
        \includegraphics[width=\textwidth]{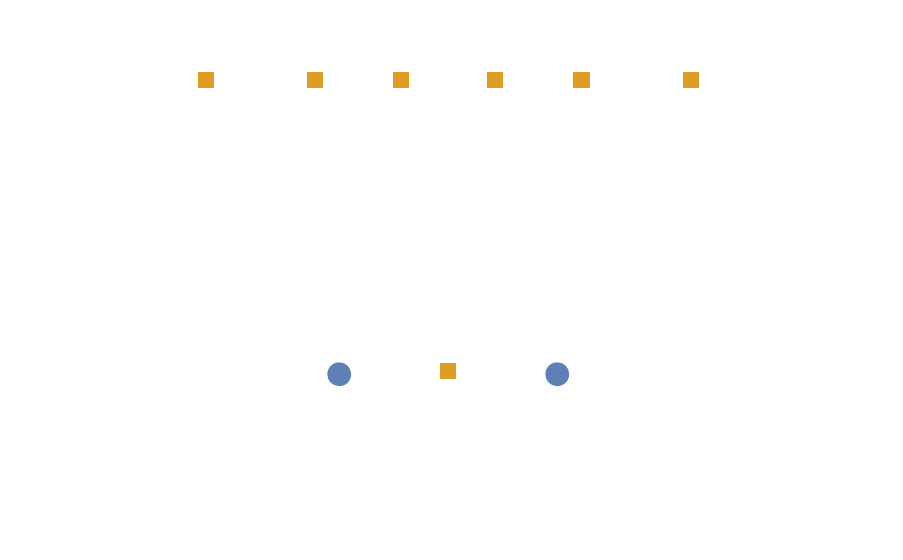}
        \caption{$\alpha\approx 7.08968$}
        \label{figure:(2,7)c}
    \end{subfigure}
    ~ 
    \begin{subfigure}[b]{0.48\textwidth}
    	    \centering
        \includegraphics[width=\textwidth]{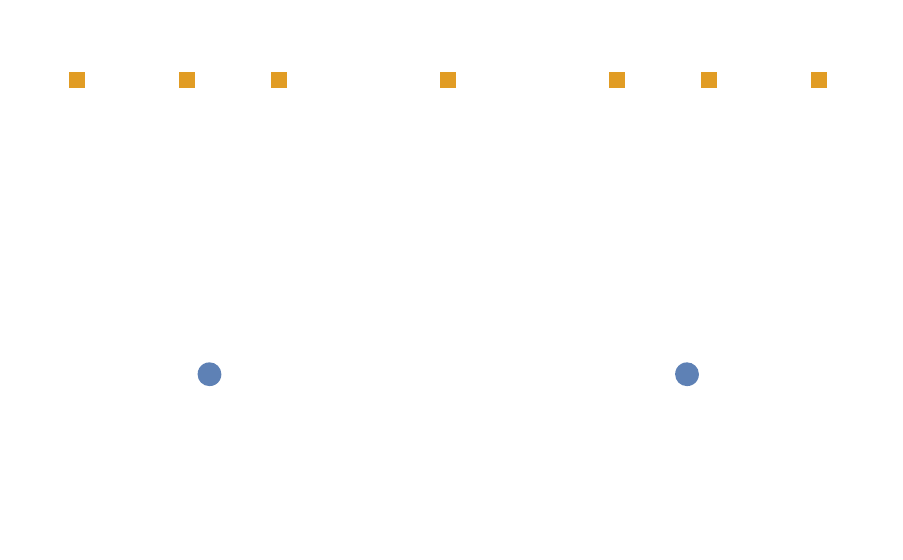}
        \caption{$\alpha\approx 69.7471$}
        \label{figure:(2,7)d}
    \end{subfigure}
    \caption{$(2,7)$ balanced configurations}
    \label{figure:(2,7)config}
\end{figure}

\section{(3,4) Configurations}
The $(3,4)$ balance equations are the smallest for which there is a balanced configuration such that the location of the nodes has no nontrivial symmetries.  Here,
\[
p_1(z)=z^3+(a_1-1)z^2-(a_1+a_0)z+a_0,\,\, p_2(z)=z^4+b_3z^3+b_2z^2+b_1z+b_0,
\]
where we assume that $p_1(1)=0$.  $Q(z)$ is a polynomial of degree six with coefficients $c_k$ such that
\[
\begin{split}
c_0&=0\Leftrightarrow -16a_0b_0-16a_1b_0+9a_0b_1=0\\
c_1&=0\Leftrightarrow -64b_0+64a_1b_0-a_0b_1-a_1b_1+36a_0b_2=0\\
c_2&=0\Leftrightarrow 144b_0-25b_1+25 a_1b_1-4a_0b_2-4a_1b_2+81a_0b_3=0\\
c_3&=0\Leftrightarrow 144a_0+81b_1-4b_2+4a_1b_2-25a_0b_3-25a_1b_3=0\\
c_4&=0\Leftrightarrow -64a_0-64a_1+36b_2-b_3+a_1b_3=0\\
c_5&=0\Leftrightarrow -16+16a_1+9b_3=0.
\end{split}
\]
As with the $(2,n)$ balance equations, we can solve iteratively, starting with $c_5=0$ down to $c_0=0$:
\[
\begin{split}
c_5=&0\Leftrightarrow b_3=\frac{16(1-a_1)}{9}\\
c_4=&0\Leftrightarrow b_2=\frac{4\left(1+36a_0+34a_1+a_1^2\right)}{81}\\
c_3=&0\Leftrightarrow b_1=\frac{16\left(1-468a_0+258a_1-261a_0a_1-258a_1^2-a_1^3\right)}{6561}\\
c_2=&0\Leftrightarrow b_0=\frac{25-70668a_0+2916a_0^2+6506a_1+69894a_0a_1-10146a_1^2}{59049}\\
&\hspace{.8in}+\frac{6606a_0a_1^2+6506a_1^3+25a_1^4}{59049}.
\end{split}
\]
This reduces the coefficients $c_1$ and $c_0$ of $Q(z)$ to fifth degree polynomials in $a_0$ and $a_1$.  Solving $c_0=c_1=0$ numerically, there are four distinct balanced configurations.  See figures \ref{figure:(3,4)config} and \ref{figure:(3,4)}.  The smallest balanced configuration with no non-trivial symmetries is shown in figure \ref{figure:(3,4)a}.

\begin{figure}[h]
    \centering
    \begin{subfigure}[b]{0.23\textwidth}
       \centering
        \includegraphics[width=\textwidth]{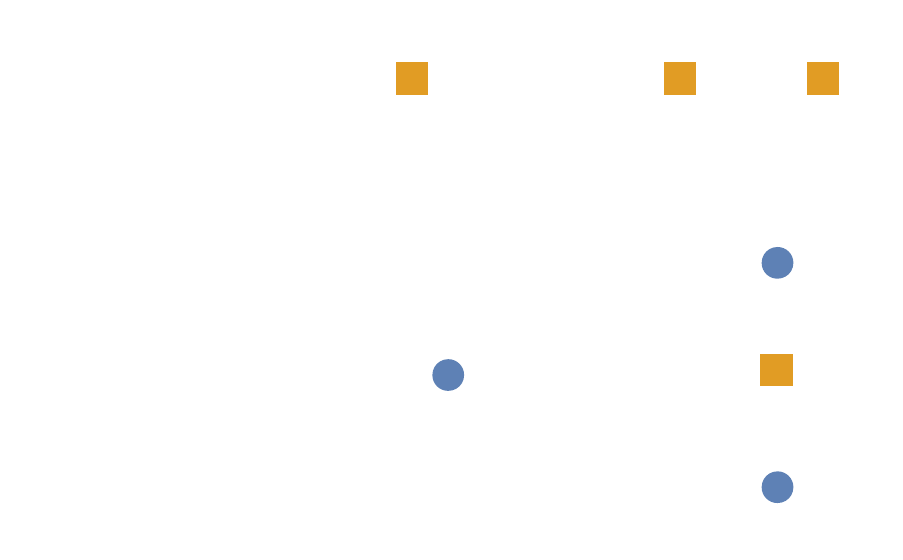}
        \caption{}
        \label{figure:(3,4)a}
    \end{subfigure}
    ~ 
    \begin{subfigure}[b]{0.23\textwidth}
    	    \centering
        \includegraphics[width=\textwidth]{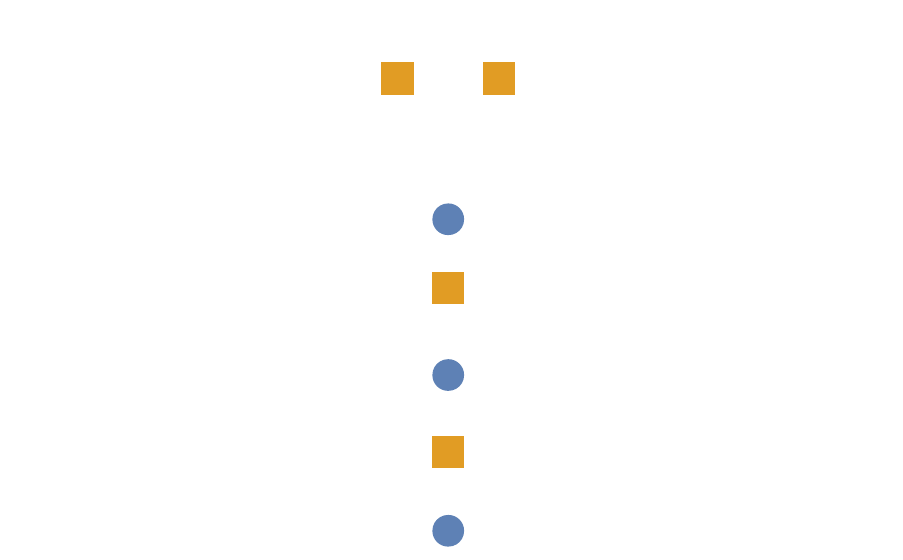}
        \caption{}
        \label{figure:(3,4)b}
    \end{subfigure}
    ~ 
    \begin{subfigure}[b]{0.23\textwidth}
    	    \centering
        \includegraphics[width=\textwidth]{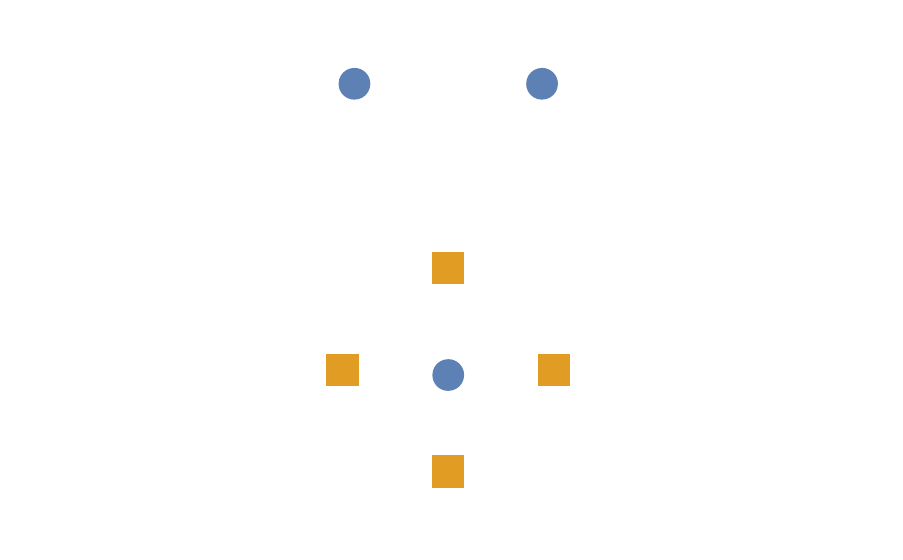}
        \caption{}
        \label{figure:(3,4)c}
    \end{subfigure}
    ~ 
    \begin{subfigure}[b]{0.23\textwidth}
    	    \centering
        \includegraphics[width=\textwidth]{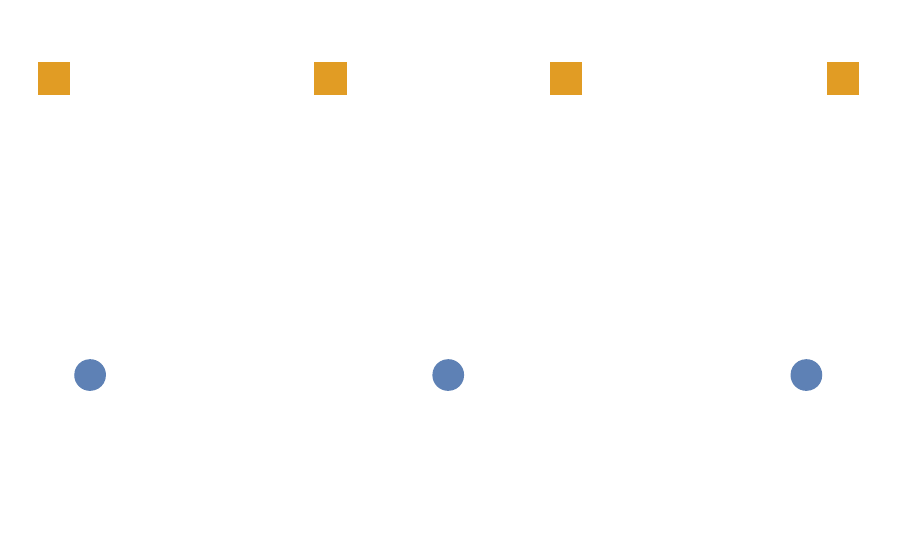}
        \caption{}
        \label{figure:(3,4)d}
    \end{subfigure}
    \caption{$(3,4)$ balanced configurations}
    \label{figure:(3,4)config}
\end{figure}

\begin{figure}[h]
	\centerline{ 
		\includegraphics[height=2in]{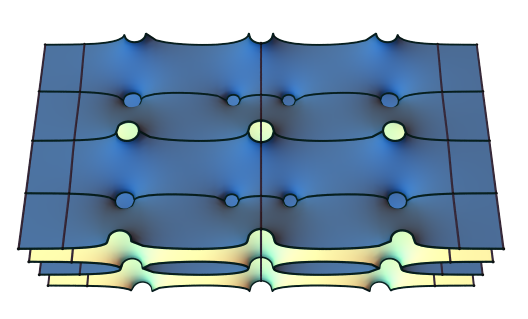}
			}
	\caption{Surface corresponding to the $(3,4)$ balanced configuration in figure \ref{figure:(3,4)d}}
	\label{figure:(3,4)}
\end{figure}

\bibliographystyle{plain}
\bibliography{balance_equations}

\begin{thebibliography}{1}

\bibitem{cw1}
P.~Connor and M.~Weber.
\newblock The construction of doubly periodic minimal surfaces via balance
  equations.
\newblock {\em Amer. J. Math.}, 134:1275--1301, 2012.

\bibitem{ka4}
H.~Karcher.
\newblock Embedded minimal surfaces derived from {S}cherk's examples.
\newblock {\em Manuscripta Math.}, 62:83--114, 1988.

\bibitem{mr3}
W.~H. Meeks~III and H.~Rosenberg.
\newblock The global theory of doubly periodic minimal surfaces.
\newblock {\em Inventiones Math.}, 97:351--379, 1989.

\bibitem{tr4}
M.~Traizet.
\newblock An embedded minimal surface with no symmetries.
\newblock {\em J. Differential Geometry}, 60:103--153, 2002.

\bibitem{tr8}
M.~Traizet.
\newblock Exploring the space of embedded minimal surfaces of finite total
  curvature.
\newblock {\em Exp. Math.}, 17:2:205--221, 2008.

\bibitem{wei2}
F.~Wei.
\newblock Some existence and uniqueness theorems for doubly periodic minimal
  surfaces.
\newblock {\em Invent. Math.}, 109:113--136, 1992.

\end{thebibliography}

\end{document}